\documentclass[12pt]{amsart}
\usepackage{amsmath}
\usepackage{amssymb}
\usepackage{epsfig}
\usepackage{wasysym}
\usepackage{graphicx}
\usepackage{tikz}
\usepackage{tikz-cd}
\usepackage{bm}
\usepackage{xcolor}
\usepackage{listings}
\numberwithin{equation}{section}
\usepackage{extpfeil}
\usepackage{array}
\usepackage{adjustbox}
\usepackage{longtable}
\usepackage{makecell}
\usepackage[all]{xy}

\input xy
\xyoption{all}

\DeclareFontFamily{U}{mathb}{\hyphenchar\font45}
\DeclareFontShape{U}{mathb}{m}{n}{
      <5> <6> <7> <8> <9> <10> gen * mathb
      <10.95> mathb10 <12> <14.4> <17.28> <20.74> <24.88> mathb12
      }{}
\DeclareSymbolFont{mathb}{U}{mathb}{m}{n}
\DeclareMathSymbol{\righttoleftarrow}{3}{mathb}{"FD}

\calclayout
\allowdisplaybreaks[3]

\theoremstyle{plain}
\newtheorem{prop}{Proposition}

\newtheorem{theo}[prop]{Theorem}
\newtheorem{coro}[prop]{Corollary}

\newtheorem{lemm}[prop]{Lemma}

\theoremstyle{definition}

\newtheorem{rema}[prop]{Remark}

\newtheorem{exam}[prop]{Example}

\def\lra{\longrightarrow}
\def\ra{\rightarrow}

\def\cC{{\mathcal C}}

\def\cL{{\mathcal L}}
\def\cM{{\mathcal M}}

\def\cP{{\mathcal P}}

\def\cV{{\mathcal V}}

\def\fD{{\mathfrak D}}

\def\fS{{\mathfrak S}}

\def\ocM{{\overline{\mathcal M}}}
\def\ocC{{\overline{\mathcal C}}}

\def\fS{{\mathfrak S}}

\def\bA{{\mathbb A}}
\def\bG{{\mathbb G}}
\def\bP{{\mathbb P}}

\def\bZ{{\mathbb Z}}

\def\bN{{\mathbb N}}
\def\bC{{\mathbb C}}

\def\rH{{\mathrm H}}
\def\rI{{\mathrm I}}

\def\diag{\mathrm{diag}}

\def\Cl{\mathrm{Cl}}
\def\Pic{\mathrm{Pic}}

\def\Mat{\mathrm{Mat}}

\def\Gal{\mathrm{Gal}}
\def\Aut{\mathrm{Aut}}
\def\Gr{\mathrm{Gr}}

\def\SL{\mathrm{SL}}
\def\GL{\mathrm{GL}}
\def\bPGL{\mathrm{PGL}}

\def\CGr{\mathrm{CGr}}

\def\Hom{\mathrm{Hom}}

\def\lim{\mathrm{lim}}

\newcommand{\bS}{\mathbb{S}}
\def\ocM{\overline{\mathcal M}}
\def\oL{\overline{L}}

\newcommand{\height}{\operatorname{ht}}

\makeatother
\makeatletter

\begin{document}

\title[Rationality of forms of $\overline{\cM}_{0,n}$]{Rationality of forms of $\overline{\cM}_{0,n}$}

\author{Brendan Hassett}
\address{Department of Mathematics\\
Brown University
Box 1917 \\
151 Thayer Street,
Providence, RI 02912 \\
USA}
\email{brendan\underline{ }hassett@brown.edu}

\author{Yuri Tschinkel}
\address{
  Courant Institute,
  251 Mercer Street,
  New York, NY 10012, USA
}

\email{tschinkel@cims.nyu.edu}

\address{Simons Foundation\\
160 Fifth Avenue\\
New York, NY 10010\\
USA}

\author{Zhijia Zhang}
\email{zhijia.zhang@cims.nyu.edu}
\address{
  Courant Institute,
  251 Mercer Street,
  New York, NY 10012, USA
}

\date{\today}

\begin{abstract}
We study equivariant geometry and 
rationality of moduli spaces of points on the projective line, for twists associated with
permutations of the points. 
\end{abstract}

\maketitle

\section{Introduction}
\label{sect:intro}

In this note, we strengthen a theorem of Florence--Reichstein \cite{FR} concerning rationality of moduli spaces. They consider {\em forms} of $\ocM_{0,n}$, i.e., varieties over nonclosed fields $F$ which are isomorphic to 
the moduli space of $n$ points on $\bP^1$ over an algebraic closure of $F$. These forms are obtained by 
twisting via Galois actions permuting the points over $F$. 
The main results of \cite{FR} are: 
\begin{itemize}
\item if $n\ge 5$ is odd, and $F$ is infinite of characteristic $\neq 2$, then every form over $F$ is rational;
\item if $n\ge 6$ is even, and $F$ has nontrivial 2-torsion in its Brauer group and contains fourth roots of unity, then there exists a form $X$ of  $\ocM_{0,n}$ 
over $F$ such that $X$ is not retract rational over $F$. 
\end{itemize}
These were inspired by a classical theorem of Enriques, Manin,  and Swinnerton-Dyer concerning rationality of twists of  $\ocM_{0,5}$, a del Pezzo surface of degree 5, over any field $F$. The proof for $n\ge 5$ uses (a twisted form of) the Gelfand-MacPherson correspondence, and techniques developed in connection with Noether's problem for twisted forms of the groups in question. 

By \cite{FR}, every form
over an infinite field $F$ is unirational over $F$. 
It is known that every form of $\ocM_{0,6}$ over $\mathbb R$ is rational \cite[Proposition 2.9]{avilov}; see Corollary~\ref{coro:cyclic} for 
generalizations.

Here, we strengthen their conclusions in two directions: we prove rationality in 
several situations not addressed in \cite{FR}. 
On the other hand, we show failure of rationality via Galois
cohomology in instances not covered by \cite{FR}, e.g.,
where the Brauer group of $F$ is trivial.

An important ingredient throughout is
a theorem of \cite{Bruno}:
$$
\Aut(\ocM_{0,n})=\fS_n, \quad  n\ge 5,
$$
acting via permutations of the $n$ points on $\bP^1$. In particular, Galois twists of 
$\ocM_{0,n}$
factor through subgroups of $\fS_n$, and there is a close link between rationality of twists and linearizability of $G$-actions on $\ocM_{0,n}$; see \cite{DR} for a general discussion of such connections. In both situations, 
there is an action of a finite group on the geometric Picard group 
$$
\Pic(\ocM_{0,n}),$$
via a subgroup of $\fS_n$.

We present several stable rationality and linearizability results, including
Propositions~\ref{prop:oddcycle} and \ref{prop:yields}
(based on the Kapranov construction) and Theorem~\ref{thm:stable}
(using torsors and quotients).
Section~\ref{sect:ratconstructs} focuses on geometric constructions.  
One rationality construction uses Schubert calculus and
the geometry of Grassmannians; Theorem~\ref{theo:ratquot} 
extends results of \cite{FR}
to small fields 
(Corollary~\ref{coro:twistfinite}) and some point configurations in higher-dimensional projective spaces (Corollary~\ref{coro:HD}).
Another relies on fibration structures; see Theorem~\ref{theo:part}. 
We close with a comprehensive discussion of the $n=6$ case (Theorem~\ref{thm:segre-rat}).

For nonrationality/nonlinearizability, we focus on situations
where the twisted moduli spaces are toric via the Losev-Manin
construction \cite{LM}.
We utilize cohomological {\bf (H1)} and {\bf (SP)}-obstructions (see Section~\ref{sect:comp}): 
In the arithmetic context, the group is replaced by the absolute Galois group of the ground field $F$ and the Picard module by the geometric Picard module. 
We focus on {\em even} $n$:

\begin{theo}[Corollary~\ref{coro:linear} and Theorem~\ref{thm:biquad}]
\label{thm:main}
For every even $n\ge 6$ there exists a subgroup $G=C_2^2\subset \fS_n$ such that  
\begin{equation*}
\label{eqn:coho}
\rH^1(G, \Pic(\ocM_{0,n})) = \bZ/2. 
\end{equation*}
In particular, 
\begin{itemize}
\item for all subgroups of $\fS_n$ containing $G$, the
corresponding action is not stably linearizable, 
\item for all fields $F$ admitting a Galois extension $L/F$ with Galois group $\Gal(L/F)\simeq G$
there exists a form $X$ of $\ocM_{0,n}$ over $F$ such that $X$ is not retract rational over $F$. 
\end{itemize}
\end{theo}

Indeed, nonvanishing group cohomology is an obstruction to (stable) linearizability, see, e.g., \cite[Corollary 2.5.2.]{BogPro}. In the context of nonclosed fields, 
one can find a twist $X$ of $\ocM_{0,n}$ over $F$ so that the 
corresponding Galois action on the geometric Picard group of $X$ factors through the prescribed action of $G$. 
This yields nontrivial Galois cohomology, which in turn obstructs retract rationality of $X$ over $F$.  
In particular, our result applies to fields $F$ with {\em trivial} Brauer group, e.g., $F=\bC(t)$.

\

\noindent
{\bf Acknowledgments:} The first author was partially supported by Simons Foundation Award 546235 and NSF grant 1929284 and 
the second author was partially supported by NSF grant 2301983. 
We are grateful to Barry Mazur and Zinovy Reichstein for comments on
this paper and its results.

\section{$\fS_n$-equivariant geometry}
\label{sect:gen}

We recall some terminology:
Let $G$ be a finite group acting regularly on a
projective variety $X$. Assume the action is generically free. The action is {\em linearizable}
if $X$ is equivariantly birational to the projectivization $\bP(V)$ of a  linear
representation $V$ of $G$ on a vector space.
It is {\em stably linearizable} if $X\times \bP^r$
-- where $G$ acts trivially on the second factor --
is linearizable.

\subsection*{Kapranov blowup}
We make use of the Kapranov blowup realization
$$
\beta_n:  \ocM_{0,n} \to \bP^{n-3}, \quad n\ge 4,  
$$
where $\beta_n$ is an iterated blowup of 
$n-1$ general points on $\bP^{n-3}$, lines through pairs of points, etc., 
see, e.g., \cite[Section 3.1]{HT-eff}. Precisely, we
regard 
$$
\bP^{n-3} = \bP(k[\fS_{n-1}]/(1,\ldots,1)),
$$
so that the $\fS_{n-1}$-action is linear. 
Boundary divisors $D_I$ are labeled by partitions 
$$
[1,\ldots, n] = I\sqcup I^c, \quad |I|, |I^c|\ge 2. 
$$

Recall that 
the Picard group $\Pic(\ocM_{0,n})$ has rank $2^{n-1}-\binom{n}{2}-1$, and  
an explicit basis is given by
$$
\{ H, E_{i_1}, E_{i_1,i_2}, \ldots, E_{i_1,\ldots, i_{n-4}} \}, 
$$
where $H$ is the (pullback of the) hyperplane class on $\bP^{n-3}$,
and the other elements are (classes of) exceptional divisors from blowups of points, lines, etc. The boundary divisors $D_{I}$ expressed in this basis are
$$
D_{i_1,\ldots,i_k,n}=E_{i_1,\ldots,i_k},\quad \{i_1,\ldots,i_k\}\subset\{1,\ldots,n-1\}, \quad k\leq n-4,
$$
and
$$
[D_{i_1,\ldots,i_{n-3},n}]=L-E_{i_1}-E_{i_2}-\ldots-E_{i_1,\ldots,i_{n-4}}-E_{i_2,\ldots,i_{n-3}}.
$$
The $\fS_n$-action on $\Pic(\ocM_{0,n})$ is best understood in terms of the natural $\fS_n$-action on the boundary divisors via permutations of indices of $D_I$. 
In particular, there is a distinguished $\fS_{n-1}\subset \fS_n$ acting via permutation of indices on $E_i$, for $i\in \{ 1,\ldots, n-1\}$. 

\

The Kapranov construction has applications to linearizability:
\begin{prop}\label{prop:oddcycle}
Suppose that $G\subseteq \fS_{n-1}$ acts on $\ocM_{0,n}$
leaving the $n$th point invariant. Then the
action of $G$ is linearizable. 

For $n=2m+1$ and $G\subseteq \fS_{2m+1}$,
the $G$-action on $\ocM_{0,n}$ is stably linearizable.

More generally, for $G\subseteq \fS_n$ leaving 
an odd cycle invariant, the $G$-action on $\ocM_{0,n}$
is stably linearizable.  
\end{prop}
\begin{proof}
The first assertion reflects the fact that 
the Kapranov morphism $\beta_n$ is $\fS_{n-1}$
invariant and the $\fS_{n-1}$-action on $\bP^{n-3}$
is linear. 
The second assertion is a special case of the third.
For the third statement, consider the universal curve
$$\ocC_{0,n}\rightarrow \ocM_{0,n}.
$$

\begin{lemm}
Let $G \subset \fS_n$ act on $\ocM_{0,n}$ by 
permutation of the marked points. Then there is 
a canonical lift of the action to the universal curve
$$\phi: \ocC_{0,n} \ra \ocM_{0,n}.$$
\end{lemm}
We prove the lemma. Interpreting $\ocC_{0,n}=\ocM_{0,n+1}$,
we have
$$
\Aut(\ocC_{0,n})=\fS_{n+1} \supset \fS_n 
\hookrightarrow \Aut(\ocM_{0,n}),
$$
with the last inclusion an equality when $n\ge 5$.
The induced action on $\Aut(\ocC_{0,n})$ is
equivariant under forgetting the $(n+1)$st point.

Returning to the Proposition, we assume that  $G$ leaves
an odd cycle invariant. Then the forgetting morphism
$\phi$ -- an \'etale $\bP^1$-bundle over $\cM_{0,n}$ --
admits a multisection of odd degree.  It must therefore
be the projectivization of a rank-two $G$-equivariant vector
bundle over $\cM_{0,n}$. 
However, we have already seen that
the $G$-action on $\ocC_{0,n}=\ocM_{0,n+1}$
is linearizable. We conclude then that $\ocM_{0,n}$
is stably linearizable.  
\end{proof}

A similar argument yields dividends for the Galois-theoretic question:

\begin{prop}
\label{prop:yields}
Let $L/F$ be a Galois extension with Galois group $\Gamma$. 
Fix a representation
$$\rho:\Gamma  \rightarrow \fS_n$$
and let ${}^{\rho}\ocM_{0,n}$ denote the corresponding
twist of $\ocM_{0,n}$ defined over $F$.
\begin{itemize} 
\item If $\rho$ factors through an $\fS_{n-1} \subset \fS_n$ then
${}^{\rho}\ocM_{0,n}$ is rational over $F$. 
\item 
If $n$ is odd then $\bP^1 \times{}^{\rho}\ocM_{0,n}$
is rational. The same holds if $\rho$ leaves an odd
cycle invariant.
\end{itemize}
\end{prop}
This gives a weaker version of \cite[Theorem~1.2]{FR}; however, our statement is valid over a finite field as well. See Remark~\ref{rema:odd+odd} below for 
a related result.

\begin{proof}
The Kapranov morphism $\beta:\ocM_{0,n} \rightarrow \bP^{n-3}$
is equivariant for $\fS_{n-1}$, which acts linearly on the 
target. Thus it descends to
$${ }^{\rho}\ocM_{0,n} \stackrel{\sim}{\rightarrow} \bP^{n-3}$$
over $F$, proving rationality.
For the second assertion, the Kapranov construction yields
$${}^{\rho}\ocC_{0,2m+1} \stackrel{\sim}{\rightarrow} \bP^{2m-1};$$
moreover 
$${}^{\rho}\ocC_{0,2m+1} \rightarrow {}^{\rho}\ocM_{0,2m+1}$$ 
is a $\bP^1$-bundle over a Zariski open
subspace of the base. (The generic fiber is a smooth genus zero curve with a cycle of odd degree.) In particular, 
$\bP^1 \times {}^{\rho}\ocM_{0,2m+1}$ is rational over $F$.
\end{proof}

\begin{exam}
Let $\fS_n$ act on $\ocM_{0,n}$, for $n\ge 5$. This 
action is not linearizable since $\fS_n$ does not act linearly and generically freely on $\bP^{n-3}$. Indeed, the smallest
faithful representation of $\fS_n$ has dimension $n-1$. When $n=p$ is a prime, then even the action of the Frobenius subgroup 
$\mathfrak F_p=\mathrm{Aff}_1(\mathbb F_p)\subset \fS_p$ 
is not linearizable, for the same reason. 
\end{exam}

\subsection*{The Losev-Manin construction}
This construction \cite{LM}, \cite[Section~6.4]{HassettWeights}
is a distinguished factorization
$$
\beta_n:\ocM_{0,n} \rightarrow \oL_n \rightarrow \bP^{n-3},
$$
where we blow up linear subspaces spanned by just
$(n-2)$ points in linear general position.
(Note that our indexing of $\oL_n$ differs from \cite{LM}.)
The first arrow contracts the boundary divisors 
$$D_{i_1,\ldots, i_k, (n-1), n}, \{i_1,\ldots,i_k\} \subset \{1,\ldots,n-2\}, \quad 
k \le n-5,
$$
by allowing points indexed by 
$$\{1,\ldots,n-2\} \setminus \{i_1,\ldots,i_k\}$$
to coincide.

We record some properties:
\begin{itemize}
\item{$\oL_n$ is toric \cite[Section 2.6]{LM};}
\item{the Losev-Manin construction is equivariant under
$\fS_{n-2}\times \fS_2 \subset \fS_n$, realized as permutations of
$\{1,\ldots,n-2\}$ and $\{n-1,n\}$ \cite[Theorem 2.5(b)]{LM}.}
\end{itemize}
The constructions of Losev-Manin give an explicit realization
of the torus $\mathsf T$ and its character module 
${\mathfrak X}^*({\mathsf T})$. Let $P$ denote the permutation
module for $\fS_{n-2}$ associated with the first $n-2$ letters 
and $L$ the non-trivial rank-one module for $\fS_2$ corresponding to $n-1$ and $n$. We regard these as modules for $\fS_{n-2}\times \fS_2$. Consider the short exact sequence
$$0 \rightarrow P_0 \rightarrow P \rightarrow \bZ \rightarrow 0$$
associated with summing over the $n-2$ letters. Then
we have 
\begin{equation}\label{eq:LMchar}
\mathfrak{X}^*(\mathsf{T}) = L \otimes P_0.
\end{equation}
Indeed, we may describe the open torus orbit in $\oL_n$ in geometric terms:
We identify the points $n-1$ and $n$ as $0$ and
$\infty$ and the first $n-2$ points as elements of
$$\Hom(P,\bP^1 \setminus \{0,\infty\})=\Hom(P,{\mathsf T}_L),$$
where ${\mathsf T}_L$ is the rank-one torus associated with $L$.
To get moduli,  we quotient out by the diagonal action of ${\mathsf T}_L$.

We record one last observation: Consider the Kapranov
blowups associated with points $n-1$ and $n$:
$$\beta_n[n-1],\beta_n[n]: \ocM_{0,n} \rightarrow \bP^{n-3}.$$
These two maps are related by an elementary Cremona transformation
$$\mathrm{Cr}:\bP^{n-3} \stackrel{\sim}{\dashrightarrow} \bP^{n-3}$$
associated with the points indexed by $\{1,\ldots,n-2\}$.
This is equivariant for the $\mathsf{T}$-actions and
we obtain a birational contraction
$$\oL_n \rightarrow \operatorname{Graph}(\mathrm{Cr}).$$

We summarize this as follows:
\begin{prop} \label{prop:twopointtoric}
Consider a twist of $\ocM_{0,n}$ associated with a subgroup
of $\fS_n$ leaving a pair of points invariant. 
This variety is necessarily toric, realized as a twist
of the Losev-Manin space.
\end{prop}
This applies in both equivariant and Galois-theoretic situations.

\subsection*{The Gelfand-MacPherson correspondence}
Our main source is Kapranov \cite{kapranov}. 

Let $\Mat(2,n)$ denote the $2\times n$
matrices. The group $\GL_2$ acts via 
multiplication from the left 
$$
\mathrm A\cdot \mathrm M \mapsto \mathrm A\mathrm M
$$
and the
torus $\mathsf T=\bG_m^n$ acts via multiplication from
the right
$$\mathrm M\cdot \mathsf T \mapsto  \mathrm M\mathsf T, \quad \mathsf T=\operatorname{diag}(t_1,\ldots,t_n).$$
Considering the action by the product $\GL_2 \times \bG_m^n$, with the elements
$$
\left( t^{-1}\, \rI_2, \diag(t,t,\ldots,t)\right)
$$
in the kernel, we obtain a faithful action
of the quotient group 
$$
(\GL_2 \times \bG_m^n)/\bG_m.
$$

We have an exact sequence
$$1 \rightarrow \mu_2 \rightarrow
\SL_2 \times \bG_m^n \rightarrow 
(\GL_2 \times \bG_m^n)/\bG_m \rightarrow 1,
$$
where 
$$
\mu_2= \left( -\rI_2,\diag(-1,-1,\ldots,-1)\right).
$$
The invariant theory quotient is
$$ \SL_2 \backslash \Mat(2,n) = \CGr(2,n),$$
the cone over the Grassmannian $\Gr(2,n)$ in its
Pl\"ucker imbedding.  
The residual action of $\bG_m^n$ on this
cone has generic stabilizer $\mu_2$; the
action on the Grassmannian has generic stabilizer
$\bG_m=\diag(t,t,\ldots,t).$
On the other hand, the geometric invariant
theory quotient 
$$\Mat(2,n) /\!/ \bG_m, \quad 
\bG_m=\diag(t,t,\ldots,t)$$
yields $(\bP^1)^n$ with factors induced by
the columns of the matrix.  
The residual $\SL_2$ acts on this product
with the distinguished linearization introduced above, which is $\fS_n$-symmetric.
Again, this action fails to be faithful, 
as $\mu_2 \subset \SL_2$ acts trivially.

\

The Gelfand-MacPherson construction yields isomorphisms
\begin{equation} \label{KapIso}
\left(\CGr(2,n) \setminus \{0\}\right) / \bG_m^n \stackrel{\sim}{\lra}
\SL_2 \backslash \! \backslash (\bP^1)^n,
\end{equation}
where both sides are interpreted as GIT quotients
\cite[2.4.7]{kapranov}. Note that we have
numerous choices for how to linearize the actions
on the left- and right-hand sides, reflecting
linearizations of the torus action and ample
line bundles on the product; Kapranov's result
makes clear how to identify these choices.
Let $X_n$ denote the quotient
arising from the $\fS_n$-symmetric linearization.

Recall that the stable and strictly semistable loci
on $(\bP^1)^n$ are easily identified 
\begin{equation}
(p_1,\ldots,p_n) \text{ stable if there is
no point with multiplicity }\ge \frac{n}{2} \label{eqn:stab}.
\end{equation}
It is semistable
if all points
have multiplicity $\le \frac{n}{2}$.
For odd $n$, stable and semistable coincide; for 
even $n=2m$, collections of points where $m$
indices coincide are strictly stable, with closed
orbits consisting of collections where
$$p_{i_1}=\cdots = p_{i_m}, \quad p_{i_{m+1}}=\cdots= p_{i_{2m}}, \quad \{i_1,\ldots,i_{2m}\}=\{1,\ldots,2m\}.
$$
In particular, $X_{2m},m\ge 3$ has $\frac{1}{2}\binom{2m}{m}$ distinguished singular points over which the orbits are identified.  


The stable loci on the Grassmannian $\Gr(2,n)$ for the
action of $\bG_m^n \cap \SL_n$ may be described 
as well:
Choose a basis diagonalizing the torus action and let 
$(A_{ij}), 1\le i<j\le n$ denote the associated Pl\"ucker
coordinates. The point
$(A_{ij})$ is stable if there are
\begin{enumerate}
\item{no index $i$ with $A_{ij}=0$ for every $j$; and}
\item{no subset $I \subset \{1,\ldots,n\}$
with $|I|\ge \frac{n}{2}$ and $A_{ij}=0$ for all $i,j \in I$.}
\end{enumerate}

These descriptions yield an $\fS_n$-equivariant stratified
blowup \cite[0.4.3,4.1.8]{kapranov}
$$
\beta:\ocM_{0,n} \rightarrow X_n.
$$
This blows down all the boundary divisors $D_I$
except those where $|I|$ or $|I^c|=2$. The divisors $D_I$ with $2|I|=n$ are
collapsed to the distinguished singular points
$\Sigma \subset X_{2m}$ where $m=|I|$ and $n=2m$.

\

The Gelfand-MacPherson construction is a powerful
tool for computing class groups.
The induced homomorphism 
\begin{equation} \label{betaclass}
\beta_*:\Pic(\ocM_{0,n})=\Cl(\ocM_{0,n})
\rightarrow \Cl(X_n)
\end{equation}
is surjective because $\beta$ is a fibration
away from the distinguished singular points.
Thus we get an exact sequence
\begin{equation}
\label{eqn:NQ} 
0\rightarrow N \rightarrow M \rightarrow
Q \rightarrow 0,
\end{equation}
where 
$$
N=\ker(\beta_*), \quad M=\Pic(\ocM_{0,n}).
$$
In particular, $N$ is generated by the 
$D_I$ where $|I|,|I^c| \neq 2$. We can easily 
compute $Q$ is well.  Write 
$$
\mathfrak{X}^*(\bG_m^n) = \bZ g_1 + \cdots + \bZ g_n,
$$
so the quotient acting faithfully on the $\CGr(2,n)$
has characters 
$$\{\sum a_i g_i: a_i \in \bZ, \sum a_i \equiv 0 \pmod 2\}.$$
These give rise to line bundles on $X_n \setminus \Sigma$
and divisor classes on the full space.  Thus
we deduce that
$$Q \subset \bZ[\fS_n/\fS_{n-1}]$$
as an index-two subgroup. Note that the element
$g_{i_1}+g_{i_2}, i_1\neq i_2$ corresponds to the boundary
divisor $D_{i_1i_2}$; indeed, this locus is cut out by the $2\times 2$ determinant on 
$\bP^1_{i_1}\times \bP^1_{i_2}$.
Since $Q$ is an index-two subgroup of a permutation
module, we have
\begin{equation}
\label{QH1}
\rH^1(G,Q) = 0 \text{ or } \bZ/2\bZ 
\quad \text{  and  } \quad
\rH^1(G,M) = 0 \text{ or } \bZ/2\bZ.
\end{equation}

When $n$ is odd, i.e., $n=2m+1$, then 
$X_{2n+1}$ is nonsingular, 
$$\Pic(X_{2m+1})=\Cl(X_{2m+1}),$$
and $\beta$ is the iteration of a sequence of blowups along smooth disjoint centers. 
Precisely, we blow up the strata where $m$ points coincide, then where $m-1$ points coincide, etc.~(see \cite[\S 8]{HassettWeights}); this is naturally equivariant
under the $\fS_{2m+1}$-action. By the blowup formula \cite[Prop.~6.7]{Fulton},
we have
$$\Pic(\ocM_{0,2m+1}) = \Pic(X_{2m+1})
\oplus \{\text{free group on the exceptional divisors}\}.$$
We summarize this in algebraic terms: 
\begin{prop}
\label{prop:stable}
For odd $n=2m+1$, the exact sequence 
(\ref{eqn:NQ}) splits $\fS_{2m+1}$-equivariantly:
$$M \simeq N \oplus Q.$$
\end{prop}

On the other hand, for $n$ even, e.g., $n=6$, there are examples of $G\subset \fS_n$ such that the sequence does not split equivariantly, since in those cases $\rH^1(G,Q)\neq 0$ while $\rH^1(G,M)=0$
(see Example~\ref{exam:n=6}).

\

We return to the isomorphism (\ref{KapIso}) over
nonclosed fields.
Up to this point, we have been working with
schemes but this is compatible with the
$\mu_2$-gerbe structure over the dense open
subset where this is the full stabilizer.
When $n=2m$ the stabilizers may be larger,
e.g., where the sequence in $(\bP^1)^{2m}$ 
consists of $m$ copies of a pair of points conjugate over a quadratic extension. 
In the cone over the Grassmannian, 
$2\binom{m}{2}=m^2-m$ coordinates vanish and
the $m^2$ remaining coordinates are equal
to the determinant of the conjugate pair.  

We can apply the same analysis to nonsplit 
actions. This includes working over nonclosed
fields, where the $n$ points are a Galois
orbit, or in the equivariant context, where the $n$ points are invariant under the action
of a finite group.  In the former situation,
over a ground field $F$ of characteristic zero,
let $E/F$ be an \'etale algebra of degree $n$
classified by a representation of the
Galois group $\Gamma_F \rightarrow \fS_n$.
We replace the group $(\GL_2 \times \bG_m^n)/\bG_m$ with 
$(\GL_2 \times R_{E/F}\bG_m)/\bG_m$
and $(\bP^1)^n$ with $R_{E/F} \bP^1$
(see \cite[\S 4]{FR}). Note however
that twisting $\Mat(2,n)=\bA^{2n}$ yields
a variety isomorphic to $\bA^{2n}$, albeit
with an action of a nonsplit torus.  

The $\mu_2$-gerbe has an explicit geometric 
interpretation along $\cM_{0,n}$: 
It is encoded by the universal family
$$\phi: \cC_{0,n} \ra \cM_{0,n},$$
a conic fibration, in general.  

\section{Rationality constructions}
\label{sect:ratconstructs}

In this section, we work over an arbitrary field $F$, and we let $\Gamma$ be the absolute Galois group of $F$.  

\subsection*{Schubert calculus background}
Our reference is \cite{Kly85}.

Consider the Grassmannian $\Gr=\Gr(p,p+q)$ of $p$-dimensional subspaces
of a vector space of dimension $p+q$.  The maximal torus
$\mathsf T=\bG_m^{p+q}$ acts diagonally on the vector space. 
Let $X$ be a generic orbit in $\Gr$.  

We set combinatorial notation: Consider shuffles of 
$\{1,\ldots,p+q\}$
$$I=\{i_1< \cdots < i_p \}, \quad J= \{j_1< \cdots < j_q\}.$$
For each such shuffle, record the pairs 
$(k,\ell), k=1,\ldots,p, \ell=1,\ldots,q$, such that 
$i_k>j_{\ell}$.
Write 
$$\lambda_{p+1-k} = \# \{\ell: j_{\ell} < i_k\}$$
and note that 
$$q\ge \lambda_1 \ge \cdots \ge \lambda_p.$$
Write $\lambda=(\lambda_1,\ldots,\lambda_p)$ 
and use the same notation for the associated Young diagram,
which fits into a $p\times q$ rectangle.  
The {\em height} $\height(\lambda)$ is the number of indices $i$
with $\lambda_i>0$.
Set $|\lambda|=\lambda_1+\cdots+\lambda_p$
and let $\sigma_{\lambda}$ denote the associated Schubert
cycle on $\Gr$, a class in $\rH^{2|\lambda|}(\Gr,\bZ)$.

We recall dimension formulae for representations.
Let $V$ be an $n$-dimensional vector space
and $\lambda=(\lambda_1,\ldots,\lambda_n)$ 
a partition of $|\lambda|$ as above;
in particular, $n \ge \height(\lambda)$.
The Schur functor $\bS_{\lambda}(V)$ is
a representation of $\SL(V)$ with dimension
\cite[Theorem 6.3, Exercise 6.4]{FulHar}:
\begin{align*}
d_n(\lambda):=\dim \bS_{\lambda}(V) &= \prod_{1\le i<j\le n} 
\frac{\lambda_i - \lambda_j + j-i}{j-i} \\
& = \prod_{(a,b)} \frac{n-a+b}{h_{ab}},
\end{align*}
where $a=1,\ldots,n$ labels the rows of $\lambda$
(from top to bottom), $b$ labels the columns
(from left to right), and $h_{ab}$ labels the ``hook length''.
This is defined as the
number of boxes immediately below and to the right 
of a given box, including the box. 
For $n<\height(\lambda)$ we set $d_n(\lambda)=0$. 

For example, when $\lambda = (\lambda_1,\lambda_2,0,\ldots)$ and
$n\ge 2$,
\begin{align*}
&d_n(\lambda_1,\lambda_2)=  \\
&\frac{(n-1+1)\cdots(n-1+\lambda_1)}{1 \cdots (\lambda_1-\lambda_2)(\lambda_1-\lambda_2+2)\cdots(\lambda_1+1)}
\frac{(n-2+1)\cdots (n-2+\lambda_2)}{1 \cdots \lambda_2} \\
& \\
&\hskip 1.7cm = \binom{n-1+\lambda_1}{\lambda_1}\binom{n-2+\lambda_2}{\lambda_2}
\frac{\lambda_1-\lambda_2+1}{\lambda_1+1}.
\end{align*}
For instance,
$$d_n(2,1)=\frac{(n+1)n(n-1)}{3}, \quad n\ge 1.$$

Another combinatorial quantity is
$$m_k(\lambda) := \sum_{i=0}^k (-1)^i \binom{|\lambda|+1}{i}d_{k-i}(\lambda).$$
If $\lambda$ has height $k$ then $m_k(\lambda)=d_k(\lambda)$, 
as the terms in the sum with $i>0$ are zero.

We record a fact that we will use repeatedly in examples: 
\begin{prop}
Fix an integer $d\ge 0$. 
If $f(x)$ is a polynomial of degree $\le d$ then
the $(d+1)$th iterated difference
$$\sum_{i=0}^{d+1} (-1)^i \binom{d+1}{i}f(x-i)=0.$$
\end{prop}
When $\lambda = (\lambda_1,\lambda_2,0,\ldots)$ we have:
\begin{align*}
&m_k(\lambda_1,\lambda_2)=  \\
&\sum_{i=0}^k (-1)^i \binom{\lambda_1+\lambda_2+1}{i} \!\!
\binom{k-i-1+\lambda_1}{\lambda_1} \!\!\binom{k-i-2+\lambda_2}{\lambda_2}
\frac{\lambda_1-\lambda_2+1}{\lambda_1+1}. 
\end{align*}
For instance, when $\lambda_1=2$ and $\lambda_2=1$ we have
\begin{align*}
m_k(2,1)&=\sum_{i=0}^k (-1)^i \binom{4}{i}\frac{(k-i+1)(k-i)(k-i-1)}{3} \\
&=2\left(\!\binom{k+1}{3}-4\binom{k}{3}+6\binom{k-1}{3}-4\binom{k-2}{3}+\binom{k-3}{3}\!\right) \\
&=\begin{cases} 2 & \text{ if }k=2, \\
               0 & \text{ if }k\ge 3.
       \end{cases} 
\end{align*}
For general $\lambda_1$ and $\lambda_2$,
$$m_2(\lambda_1,\lambda_2)=\lambda_1-\lambda_2+1$$
and
$$m_3(\lambda_1,\lambda_2)=\frac{\lambda_1(\lambda_2-1)(\lambda_1-\lambda_2+1)}{2}.$$

\begin{theo} \cite[Theorem~5]{Kly85} \label{theo:KF}
If $X$ is the generic torus orbit in $\Gr=\Gr(p,p+q)$ 
and $\lambda$ is a partition with $|\lambda|=p+q-1$
then
$$[X]\cdot \sigma_{\lambda} =  m_p(\lambda).$$
\end{theo}

For example, take $p=2$. For $q=2$ 
$$[X]\cdot \sigma_{21}=2$$
and when $q=3$ we have
$$[X]\cdot \sigma_{22}=1, \quad [X]\cdot \sigma_{31}=3.$$
For general $q$, we have $\lambda_1 \ge \lambda_2=q+1-\lambda_1 \ge 0$,
i.e.,
$$
\frac{q+1}{2} \le \lambda_1 \le q+1.
$$
Here we have
$$[X]\cdot \sigma_{\lambda_1\,q+1-\lambda_1} = 2\lambda_1 - q;$$
in particular, when $q=2m-1$ and $\lambda_1=m$
we find
$$[X]\cdot \sigma_{m\,m}=1.$$

\begin{rema}The signs in the formula for $m_k(\lambda)$ 
obscure the positivity of the result. An alternate formula
\cite[Theorem~5.1]{BF} makes this clearer:
$$
[X] = \sum_{\lambda \subset (q-1)^{p-1}}
\sigma_{\lambda} \sigma_{\widetilde{\lambda}},
$$
where $\widetilde{\lambda}$ is the complement to $\lambda$
in the rectangle $(q-1)^{p-1}$:
$$\lambda=(\lambda_1,\ldots,\lambda_{p-1}), \quad
\widetilde{\lambda}=(q-1-\lambda_{p-1},\ldots,q-1-\lambda_1).$$
We refer the reader to \cite{Lian} for the combinatorics
directly relating these formulas. 
\end{rema}
This extends to general $p\in \bN$:
\begin{prop} 
\label{prop:degreeone} 
Let $V$ be a vector space with $\dim(V)=mp+1$
so that 
$$q=(m-1)p+1 \quad \text{ and }\quad (p-1)(q-1)=(m-1)(p-1)p.
$$
Consider
the coefficient of 
$$\sigma_{\underbrace{(m-1)(p-1)\ldots (m-1)(p-1)}_{p \text{ times }}}$$
in the expansion of $[X]$ in $\rH^{2(p-1)(q-1)}(\Gr(p,p+q))$.
This equals $1$, i.e.,
$$
[X]\cdot \sigma_{\underbrace{m \ldots m}_{p \text{ times }}}=1.
$$
\end{prop}
Indeed, this follows from Klyachko's formula (Theorem~\ref{theo:KF}) and 
$$m_p(\underbrace{m,\ldots,m}_{p \text{ times}})
=d_p(\underbrace{m,\ldots,m}_{p \text{ times}})=1.$$

\begin{exam}
When
$\dim(V)=3m+1$ the generic orbit $X$ for the action
of $T$ on $\Gr(3,V)$ has codimension $3(3m-2)-3m=6(m-1)$
and 
$$[X]\cdot \sigma_{m\, m\, m}=m_3(m,m,m)=d_3(m,m,m)=1.$$
This is not the case when $\dim(V)=3m+2, m>1$, e.g.,
$$
[X]=
10 \sigma_{5,3} + 8 \sigma_{5,2,1} + 15 \sigma_{4,4}+ 
15 \sigma_{4,3,1} + 6 \sigma_{4,2,2} + 3\sigma_{3,3,2}.
$$
\end{exam}

\subsection*{Grassmann geometry and rationality}

\begin{theo} \label{theo:ratquot}
Let $\mathsf T$ be a maximal torus -- possibly nonsplit - for $\SL_{pm+1}$
over a field $F$. Take $\Gr(p,V)$ for $\dim_F(V)=pm+1$ with
the resulting $\mathsf T$-action.  Choose a subspace
$W \subset V$ with 
$$
\dim_F(W)=(p-1)m+1
$$ 
and transverse to $\mathsf T$ in the sense that $\Gr(p,W) \subset \Gr(p,V)$
meets some stable $\mathsf T$-orbit properly.
Then $\Gr(p,W)$
is a rational section of the quotient
$$
\Gr(p,V) \stackrel{\sim}{\dashrightarrow} \Gr(p,V)/\mathsf T.
$$
Thus if $\Gr(p,W)$ is rational, linearizable, or stably
linearizable then the same holds true of the quotient.  
\end{theo}
\begin{proof}
The stability assumption guarantees that the quotient map is defined
over a non-empty open subset of $\Gr(p,W)$. Properness of the
intersection -- which has degree one by Proposition~\ref{prop:degreeone}! -- implies $\Gr(p,W)$ is 
mapped birationally to the quotient.
\end{proof}

\begin{prop} \label{prop:transverse}
Retain the notation of Theorem~\ref{theo:ratquot}.

If $F$ is infinite then $\Gr(p,V)$ admits a codimension-$m$
subspace $W\subset V$ satisfying the transversality
condition.

If $F$ is finite and $p=2$ then $\Gr(2,V)$ admits a stable
$F$-rational point.

If $F$ is arbitrary and $p=2$ then for each stable point
there exists a subspace $W$ satisfying the transversality assumption.
\end{prop}
Combining with Theorem~\ref{theo:ratquot} gives a generalization
of the results of \cite{FR}:
\begin{coro} \label{coro:twistfinite}
Let $F$ be a finite field and $\rho$ a representation of
its Galois group in $\fS_{2m+1}$. Then ${ }^{\rho}\ocM_{0,2m+1}$
is rational over $F$.
\end{coro}
We also obtain analogs in higher dimensions:
\begin{coro} \label{coro:HD}
Let $m\ge 1$ and $p\ge 2$ be integers.
Consider the moduli space of $pm+1$ points in $\bP^{p-1}$
up to projective equivalence. Let $X$ be a variety obtained
by twisting via permutations of the points, over an infinite field $F$.   
Then $X$ is rational.  
\end{coro}
\begin{proof}[Proof of Proposition~\ref{prop:transverse}]
Assume $F$ is infinite; here we use \cite[\S 1.2]{kapranov}.
While Kapranov assumes the ground field has characteristic
zero, the toric constructions and interpretation of $\ocM_{0,n}$ 
as a Chow quotient for the $\bPGL_2$-action are valid in positive characteristic \cite{GGchow}. 

The Grassmannian is rational over $F$ so its $F$-rational points
are Zariski dense.  We note that the torus
action determines a collection of $\overline{F}$-subspaces
$$V_I \subset V, \quad \emptyset \neq I=\{i_1,\ldots,
i_r \} \subset \{0,\ldots, mp \},
$$
spanned by eigenvectors of the torus. Consider the 
$$
W \in \Gr(mp+1-m,mp+1)
$$ 
meeting some of these improperly,
i.e., 
$$\dim(W \cap V_I) > \dim(W) + \dim(V) - \dim(V_I).$$
This is a Zariski closed proper subset of the Grassmannian,
defined over $F$; its complement has $F$-rational points.
Given such a subspace $W\subset V$, choose 
$$w \in \Lambda \subset W, \quad \dim(\Lambda)=p,
$$ 
defined over $F$, with $w$ not contained
in any of the $V_I\subsetneq V$ and 
$\Lambda$ meeting all the $V_I$ 
properly. Thus $\Lambda$ is stable for the torus action
and the torus orbit of $\Lambda$ meets $\Gr(p,W)$ transversally there.

Now assume that $F$ is finite and $p=2$. We use the stability
criterion (\ref{eqn:stab}) for points on $\bP^1$
and Kapranov's analysis of the Gelfand-MacPherson correspondence.
Here the Galois action $\rho$ on the $2m+1$ points is encoded by a 
single element $\sigma \in \fS_{2m+1}$. Express $\sigma$ as a product
of $r$ disjoint cycles of lengths $\ell_i$ with 
$$\ell_1+\cdots+\ell_r=2m+1, \quad \ell_1\ge \ell_2 \ge \cdots \ge \ell_r.$$
Only $\ell_1$ can possibly be greater than $m$;
if $\ell_1 \le m$ then we have $r\ge 3$.
When $\ell_1>m$, choose a configuration of $\ell_1$ points
defined over a degree-$\ell_1$ extension of $F$. Allow the
remaining points to all coincide. We turn to the situation where
$\ell_1 \le m$.  If $r=3$ then we allow $\ell_1$ points
to coincide with $[0,1]$, $\ell_2$ points to coincide with $[1,0]$,
and $\ell_3$ points to coincide with $[1,1]$. We may therefore
assume that $r\ge 4$ and work inductively on $r$. There
exists two indices, say $\ell_3$ and $\ell_4$, whose sum is
less than $m$. Use this to ``degenerate'' to a new partition of $2m+1$,
refined by $(\ell_1,\ldots,\ell_r)$ but 
of length $r-1$, all of whose entries are less than $m$.  
For example, we could take 
$$(\ell_1,\ell_2,\ell_3+\ell_4,\ell_5,\ldots,\ell_r).$$
Continuing in this way, we generate a partition
$$\{1,2,\ldots,r\} = A \sqcup B \sqcup C$$
such that 
$$\sum_{a\in A} \ell_a, \sum_{b\in B} \ell_b, \sum_{c\in C} \ell_c
\leq m.$$
Let points coincide in three groups according to this coarsening of our original partition, the first group to $[0,1]$, the second
to $[1,0]$, and the third to $[1,1]$.

Assume $p=2$ and $F$ is arbitrary.
We continue to assume that $\Lambda \subset V$ 
is a two-dimensional subspace that is stable in
the sense of Geometric Invariant Theory.
Let $\mathbf{T}_{2m}$ denote the tangent 
space to the torus orbit at $\Lambda$
$$\mathbf{T}_{2m} \subset \Hom(\Lambda,V/\Lambda),$$
an $2m$-dimensional subspace of the tangent space
to $\Gr(2,V)$ at $\Lambda$.
We claim there exists a subspace
$$\Lambda \subset W \subset V,$$
where $W$ has codimension $m$ in $V$, such that the composition
$$\mathbf{T}_{2m} \subset \Hom(\Lambda,V/\Lambda) \twoheadrightarrow
\Hom(\Lambda,V/W)$$
has full rank $2m$. Since the latter space is the normal directions
to $\Gr(2,W)$ at $\Lambda$, this will yield transversality. 

We record some basic geometry:
\begin{lemm}
There is a distinguished orbit 
$$\bP^1 \times \bP^{m-2} \simeq \bP(\Lambda^*) \times \bP(V/\Lambda) \subset 
\bP(\Hom(\Lambda,V/\Lambda))$$
invariant under automorphisms of $\Gr(2,V)$ fixing $[\Lambda]$.

The subspace $\bP^{2m-1}\simeq \bP(\mathbf{T}_{2m})$ cuts
out the graph of a rational normal curve 
\begin{align*}
\varrho: \bP^1_{s_0,s_1} &\hookrightarrow \bP^{2m-2}_{x_0,\ldots,x_{2m-2}} \\
[s_0,s_1] &\mapsto [s_0^{2m-2},\ldots,s_1^{2m-2}].
\end{align*}
In these coordinates, the rational normal curve has equations
$$s_0x_{i+1}=s_1x_i, \quad i=0,\ldots,2m-1.$$

Let $\Gamma \subset \bP^1$ denote the length-$(2m+1)$ subscheme 
that is the image of the eigenvectors for 
$\mathbf{T}_{2m}$ under $V^* \twoheadrightarrow \Lambda^*$.
Then $\varrho$ realizes the Gale transform for $\Gamma \subset
\bP^1$ as a subscheme of $\bP^{2m-2}$ contained in a 
rational normal curve.
\end{lemm}
The first assertion reflects the fact that the parabolic
subgroup of $\bPGL_{2m+1}$ fixing $[\Lambda]$ 
has semisimple part $(\GL_2\times \GL_{2m-1})/\bG_m$. Note
that the unipotent part acts trivially on the tangent space.
The second assertion is true for the generic codimension-$(2m-2)$ linear slice of $\bP^1 \times \bP^{2m-1}$. Of course, one
has to show that this applies in our siutation! This follows
from the third assertion, a special
case of \cite[Corollary 3.2]{EisPop} -- the first application
following the statement.  This completes the proof of the lemma.

Returning to the proof of the Proposition, we may take 
$W$ as the subspace given by
$$\{x_{2j}=0, j=0,\ldots, m-1\},
$$ 
where we interpret $x_j \in (V/\Lambda)^*a$.
It is clear that the products
$$\{s_ix_{2j}, i=0,1, j=0,\ldots, m-1\}$$
have the desired spanning property; the elements
$$s_0^{2m-1},\ldots,s_1^{2m-1}$$
are a basis for bilinear forms of degree $2m-1$.
\end{proof}

\subsection*{Partitioning the points}
We start with a general construction: Let $n\ge 3$ be an integer
and $n=\ell m$ a factorization in integers $\ell, m>1$.
Suppose that $H \subset \fS_{\ell}, A \subset \fS_m$ are
subgroups.  The {\em wreath product} 
$$A\wr H = A\wr_{1,\ldots,\ell} H$$
is the semidirect product $A^{\ell} \rtimes H$ where
$$(a_1,\ldots,a_{\ell})\cdot h= (a_{h^{-1}(1)},\ldots, a_{h^{-1}(\ell)}).$$
This comes with a natural embedding 
$$\rho: A\wr H \hookrightarrow \fS_{\ell m}$$
as permutations of pairs 
$$(i,j), \quad i \in \{1,\ldots,m\}, j \in \{1,\ldots,\ell\}.$$

Now assume that $m\ge 3$. Forgetting maps yield an
equivariant morphism
$$
\phi: { }^{\rho}\ocM_{0,\ell m} 
\rightarrow \prod_H { }^{\alpha}\ocM_{0,m},
$$
where $\alpha:A \hookrightarrow \fS_m$ and 
the twisted product denotes $\ell$ copies
of the moduli space with the associated $H$-action.
The generic fiber of this morphism is irreducible
of dimension
$$(\ell m -3)-\ell(m-3)=3\ell -3.$$
It is birational to the Hilbert scheme of multidegree-$(1,\ldots,1)$ curves 
in the $H$-twisted product $\prod_H C_j$ of $\ell$ genus-zero curves.  
Geometrically, this is a compactification of the homogeneous space
$$\underbrace{\bPGL_2 \times \cdots \times \bPGL_2}_{\ell \text{ times }}/\bPGL_2$$
with the last $\bPGL_2$ embedded diagonally.  

We record some observations on the generic fiber of $\phi$:
\begin{itemize}
\item
Suppose $\ell=2$.  Geometrically, $(1,1)$ curves in $\bP^1 \times \bP^1$
are parametrized by $\bP^3$ -- the dual to the projective space containing 
the Segre embedding of $\bP^1\times \bP^1$.  Over an arbitrary field
the fiber is a Brauer-Severi threefold.
\item
Suppose that $m$ is odd.  Then the genus-zero curves $C_j$ appearing in the
twisted product
are split and -- over the extension/subgroup associated with 
$A^{\ell} \subset A\wr H$ 
-- isomorphic to $\bP^1$'s. Here the twisted product $\prod_H C_j$ is rational,
as it is isomorphic to the restriction of scalars of $\bP^1$.
\item
Now assume $\ell=2$ and $m$ odd. Here the generic fiber of $\phi$
is isomorphic to $\bP^3$ over the function field/linearizable for the full wreath product. 
\end{itemize}

\begin{exam} \label{exam:segp3}
Suppose $n=6$ and consider $G=\fS_3 \wr \fS_2 \subset \fS_6$, a subgroup of
index $10$ preserving an unordered partition
$$\{1,2,3,4,5,6\} = \{i,j,k \} \sqcup \{a,b,c\}.$$
Then the associated ${ }^{\rho}\ocM_{0,6}$ is rational/linearizable.
These actions correspond to situations where the associated Segre threefold
admits an invariant node (cf.~Theorem~\ref{thm:segre-rat} below).
\end{exam}

\begin{theo} \label{theo:part}
Let $n=2m$, with $m\ge 3$ odd. Fix a subgroup
$A \subset \fS_m$ and the diagonal subgroup
$$G:=A \times \fS_2 \subset A \wr \fS_2  \subset \fS_{2m}.$$

\begin{itemize}
\item{For each Galois representation $\rho:\Gamma \rightarrow G$
the twist ${ }^{\rho}\ocM_{0,n}$ is rational over $F$.}
\item{The $G$ action on $\ocM_{0,n}$ is stably linearizable.}
\end{itemize}
\end{theo}
\begin{proof}
We assume $\fS_m$ permutes the points with odd and even indices respectively.

We focus first on the arithmetic case. Let $L/F$ be the quadratic
extension associated with $A$. Over $L$, the generic point 
of the twisted moduli space corresponds to $\bP^1$ equipped
with reduced and disjoint
zero-cycles $Z_{odd},Z_{even}\subset \bP^1$ of length $m$. 
The parity of $m$ ensures that the underlying curve is $\bP^1$.

Note that the variety ${}^{\rho}{\ocM_{0,n}}$ is already
stably rational over $L$ by Proposition~\ref{prop:yields}.

Consider forgetting the even and odd points
$$(\pi_{odd},\pi_{even}):({ }^{\rho}{\ocM_{0,n}})_L \rightarrow
{ }^{\varpi_{odd}}\ocM_{0,m} \times { }^{\varpi_{even}}\ocM_{0,m}$$
where the Galois actions come via restriction to the even and
odd points. These actions are conjugate for the quadratic
extension $L/F$.  Descent therefore gives a morphism over $F$
$$\phi:{ }^{\rho} \ocM_{0,n} 
\rightarrow R_{L/F} ({ }^{\varpi_{odd}}\ocM_{0,m}),
$$
where the target is the restriction of scalars. The twists of
$\ocM_{0,m}$ are rational over $L$ by \cite{FR} and Corollary~\ref{coro:twistfinite}.
The restriction of scalars of a rational variety is rational.  

We claim that the generic fiber of $\phi$ is rational over the function field of the base, which implies rationality for 
${ }^{\rho}\ocM_{0,n}$ over $F$.  
This follows from the analysis above for $\ell=2$ and odd $m$.

For the equivariant case, our geometric argument shows that 
the $G$-variety $\ocM_{0,n}$ is birationally the projectivization
of an equivariant vector bundle over a stably linearizable variety
(by Proposition~\ref{prop:yields}). Note that restriction of 
scalars in the arithmetic situation corresponds to passing to an induced
representation in the equivariant context; thus stable 
linearizability is clearly preserved.  We conclude then that
$\ocM_{0,n}$ is stably linearizable.
\end{proof}

\begin{coro} \label{coro:cyclic}
Let $C_{2m}$, with $m$ odd, be a cyclic group. Then twists of $\ocM_{0,n}$
by this group are rational (in the Galois case) and stably linearizable (in the
equivariant situation).
\end{coro}
\begin{proof}
If the action has an odd orbit then this follows from 
Propositions \ref{prop:oddcycle} and \ref{prop:yields}.  
Otherwise, all the orbits are even and we may apply Theorem~\ref{theo:part}.  
\end{proof}

\begin{rema} \label{rema:odd+odd}
Similar reasoning applies for a Galois action
$$\rho: \Gamma \rightarrow 
\fS_{m_1} \times \fS_{m_2} \subset \fS_{m_1+m_2}, \quad
m_1,m_2\ge 3 \text{ odd},$$
with restricted actions $\varpi_1$ and $\varpi_2$
on the first $m_1$ points and last $m_2$ points respectively.
Proposition~\ref{prop:yields}
already gives stable rationality in this case.
The forgetting morphism
$$\phi:{ }^{\rho} \ocM_{0,m_1+m_2} \rightarrow 
{ }^{\varpi_1}\ocM_{0,m_1} \times { }^{\varpi_2}\ocM_{0,m_2}
$$
has generic fiber birational to $\bP^3$ 
by the reasoning above. Since the factors
${ }^{\varpi_i}\ocM_{0,m_i}$ are rational, 
${ }^{\rho} \ocM_{0,m_1+m_2}$ is rational as well.
\end{rema}

\section{Stable linearizability via torsors}
\label{sect:linear}

Let $G$ be a finite group and $\mathsf T$ a $G$-torus,
i.e., a torus equipped with a representation of $G$ 
on its character module $\mathfrak X^*(\mathsf T)$. 
Recall that $\mathsf T$ is stably linearizable if
$\mathfrak X^*(\mathsf T)$ is stably permutation, see, e.g., \cite[Proposition 2]{HT-torsor}.

\begin{prop}
\label{prop:stableT}
Let $U$ be a smooth quasi-projective variety
with $G$-action.  
Assume that we have a $\mathsf T$-torsor
$$\cP \rightarrow U,$$
i.e., a $\mathsf T$-principal homogeneous space
over $U$, in the category of $G$-varieties. Assume that 
\begin{itemize}
\item 
the $G$-action on $U$ is generically free,  
\item the characters $\mathfrak X^*(\mathsf T)$ are a stably permutation $G$-module, 
\item the $G$-action on $\cP$ is stably linearizable. 
\end{itemize}
Then the $G$-action on $U$ is stably linearizable.
\end{prop}
\begin{proof}
We claim there is a $G$-equivariant birational map,
$$
\begin{array}{rcccl}
\cP & &\stackrel{\sim}{\dashrightarrow} & & \mathsf{T} \times U \\
   & \searrow & & \swarrow & \\
   &          & U & & 
   \end{array}
$$
which would follow if $\cP \rightarrow U$ 
admits a $G$-equivariant rational section.  
We clearly have such a section after 
discarding the $G$-action, by Hilbert's Theorem 90.

Since $\mathsf T$ is stably permutation, a
product $\mathsf T \times \mathsf T_1$, where $\mathsf T_1$
is a permutation torus, is
isomorphic to a permutation torus and may be
realized as a dense open subset of affine space.  
It follows that we have an open embedding
$$
\begin{array}{rcccl}
\cP\times_U \mathsf T_1 & & \hookrightarrow & & \cV \\
  & \searrow  &  & \swarrow & \\
      &   &  U  &  & 
      \end{array}$$
where $\cV\rightarrow U$ is a vector bundle 
with $G$-action.  The vector bundle admits a rational section (by the No-Name Lemma) thus $\cP$ 
does as well.

We assumed that $\cP$ is stably linearizable,
i.e. $\cP\times \bG_m^r$ is linearizable
for some $r$. Thus $U \times \mathsf T \times \bG_m^r$ is as well.  
We observed that $\mathsf T$ is stably linearizable because its character module
is stably permutation, i.e. $\mathsf{T}\times \mathsf{T}_1$ is a permutation torus. Another application of the
No-Name Lemma, using the assumption that the action
on $U$ is generically free, gives that $U$ is 
stably linearizable.
\end{proof}

We recall the exact sequence \eqref{eqn:NQ}
$$
0\to N\to M\to Q\to 0,  
$$
with  $M=\Pic(\ocM_{0,n})$, $N$ an $\fS_n$-permutation module, and 
$Q$ is an index-2 submodule of the permutation module $\bZ[\fS_n/\fS_{n-1}]$. 
We record:
\begin{itemize}
    \item if $\rH^1(G,Q) =0$ for some $G\subset \fS_n$, then also $\rH^1(G,M)=0$, 
    by the long exact sequence in cohomology, 
    \item if $Q$ is a stably permutation $G$-module, then the sequence splits and 
    $\Pic(\ocM_{n,0})$ is a stably permutation module, by \cite[Lemma 1]{CTS77}. 
\end{itemize}


\begin{theo}
\label{thm:stable}
Let $G\subseteq \fS_n$ be a subgroup such that $Q$ is a stably permutation module. Then 
the $G$-action on $\ocM_{0,n}$ is stably linearizable.

Let $X$ be a form of $\ocM_{0,n}$ over $F$ such that the action of the absolute Galois group on $Q$ gives rise to a stable permutation module. 
Then $X$ is stably rational over $F$. 
\end{theo}

\begin{proof} 
For the equivariant statement, we apply Proposition~\ref{prop:stableT}.
Here $\mathsf{T}$, with character module $Q$ 
acts on $\CGr(2,n)$ (see Section ~\ref{sect:gen}). Let $V\subset \CGr(2,n)$ the open subset over
which $\mathsf{T}$ acts freely and $U \subset X_n$
the corresponding locus in the quotient, i.e., 
remove all the strictly semistable points. 
We have a torsor
    $$
    V\stackrel{\mathsf T}{\lra} U.
    $$
By \cite[Proposition 19]{HT-torsor}, the $\fS_n$-action on $\Gr(2,n)$ (and its cone)
is stably linearizable. 
 Assuming that $Q=\mathfrak X^*(\mathsf T)$ is a stable permutation module for $G\subset \fS_n$, and applying Proposition~\ref{prop:stableT}, we conclude that the $G$-action on $U$, and thus $\ocM_{0,n}$, is stably linearizable as well. 

The Galois-theoretic result is proven analogously, with
\cite[Prop.~3]{BCTSSD} playing the role of
Proposition~\ref{prop:stableT}. This is an application of the torsor formalism of \cite{CTS87}.
\end{proof}

\begin{rema}
\label{rema:7}
There exist linearizable $G$-actions on $\ocM_{0,n}$ such that the induced action on $Q$ is not stably permutation. Consider $n$ even and $G=C_2$ generated by
$\sigma:=(1,2)\cdots (n-1,n)$; we have $\rH^1(C_2,Q)\neq 0$
(see Remark~\ref{rema:H1Qnonzero}) so $Q$ is not stably permutation.
This action is equivariantly birational -- by Proposition~\ref{prop:twopointtoric} -- to an action on a torus ${\mathsf T}=\bG_m^{n-3}$. Its character module consists
of the elements of $\bZ^{n-2}$ -- the twisted permutation
module on $\{1,\ldots,n-2\}$ -- whose coordinates sum to zero
(see Equation~\ref{eq:LMchar}).
The action of $C_2$ on the twisted permutation module 
consists of $(n-2)/2$ copies of
$\left( \begin{matrix} 0 & -1 \\ -1 & 0 \end{matrix} \right)$.
Hence $\mathfrak{X}^*(\mathsf{T})$ 
decomposes as a sum of $\frac{n}{2}-2$ permutation modules 
and one invariant, a permutation module.
We conclude $\mathsf T$ is linearizable.
\end{rema}

\begin{rema}
\label{rema:stable}
By \cite[Remark 5.5]{FR}, for {\em odd} $n$, every form of $\ocM_{0,n}$
over a nonclosed field $F$
is an $F$-rational variety. {\em A priori}, this does {\em not} imply that $\ocM_{0,n}$ is (stably) linearizable for $\fS_n$. 
However, this does imply that $M$ is a stable permutation module, for the $\fS_n$-action.

For $n$ {\em odd}, we have 
\begin{equation}
    \label{eqn:split}
M\simeq N\oplus Q,
\end{equation}
as $\fS_n$-modules, by Proposition~\ref{prop:stable}. 
Since $N$ is a permutation module, for all $n$, and $M$ a stably permutation module, for odd $n$, we see that $Q$ is also stably permutation, for odd $n$. Thus,   
the $\fS_n$-action on 
$\ocM_{0,n}$ is stably linearizable, by Theorem~\ref{thm:stable}.  

The splitting \eqref{eqn:split} 
can also be seen explicitly:
Recall that under the Kapranov basis, $Q=M/N$ is generated by the image of the classes 
$$
H,\quad E_i,\quad i=1,\ldots,n-1
$$
in $M$ under the projection modulo $N$. The $\bZ$-linear map
$$
s: Q\to M,
$$
given on these generators by
$$
H\mapsto H+\sum_{\substack{I\subset\{1,\ldots,n-1\},\\|I|=\frac{n-1}2,\ldots,n-4.}}(|I|-1)\cdot E_I,
\quad \quad 
E_i\mapsto E_i+\sum_{\substack{I\subset\{1,\ldots,n-1\}, i\in I,\\|I|=\frac{n-1}2,\ldots,n-4.}}E_I.
$$
is a section of the exact sequence~\eqref{eqn:NQ}. We check that it is $\fS_n$-equivariant. Let $\tau=(1,2)$ and $\sigma=(1,\ldots,n)$. In $Q$, one has 
$$
H=D_{12}+\sum_{i=3}^nE_i
$$
and $\tau(H)=H$, $\tau(E_1)=E_2$, $\tau(E_2)=E_1$ and $\tau(E_i)=E_i$. 
Note that $s$ is $\tau$-equivariant by construction. 
Next, observe 
\begin{align*}
    s\sigma(H)&=s\left(\sigma \left(D_{12}+\sum_{i=3}^nE_i\right)\right)
    =s\left((n-3)H-(n-4)\sum_{i=2}^{n-1}E_i\right)\\
    &=(n-3)H-(n-4)\sum_{i=2}^{n-1}E_i-\sum_{\substack{I\subset\{1,\ldots,n-1\}, 1\notin I,\\|I|=\frac{n-1}{2}, \ldots, n-4.}}(n-|I|-3)\cdot E_{I}\\
    &\phantom{=}+\sum_{\substack{I\subset\{1,\ldots,n-1\}, 1\in I,\\|I|=\frac{n-1}{2}, \ldots, n-4.}}(|I|-1)\cdot E_{I}.
\end{align*}
\begin{align*}
    \sigma s(H)&=\sigma\left(H+\sum_{|I|=\frac{n-1}{2}, \ldots, n-4}\left(|I|-1\right)\cdot E_{I}\right)\\
    &=\sigma\left(D_{n-2,n-1}+\sum_{\substack{I\subset\{1,\ldots, n-3\},\\|I|=1,\ldots,n-4.}}E_{I}+\sum_{|I|=\frac{n-1}{2}, \ldots, n-4.}\left(|I|-1\right)\cdot E_{I}\right)\\
    &=E_{n-1}+\sum_{\substack{I\subset\{2,\ldots,n-2\}\\|I|=1, \ldots, n-5}}D_{I\cup \{n-1,n\}}+\sum_{i=2}^{n-2}D_{1,i}\\
    &\phantom{=}+\underbrace{\sum_{\substack{I\subset\{2,\ldots,n-1\},\\|I|=\frac{n-1}{2}-1, \ldots, n-5.}}E_{\{1\}\cup I}+\sum_{\substack{I\subset\{2,\ldots,n-1\},\\|I|=2, \ldots, \frac{n-1}{2}-1.}}(n-3-|I|)E_{I}}_{:=A}\\   
    &=(n-3)H-(n-4)\sum_{i=2}^{n-1}E_i-\sum_{\substack{I\subset\{2,\ldots,n-1\}, i \notin I,\\|I|=1,\ldots,n-4.}}E_I\\
    &\phantom{=}+\sum_{\substack{I\subset\{2,\ldots,n-2\},\\|I|=1, \ldots, n-5.}}D_{I\cup \{n-1,n\}}+A.
\end{align*}
One can then verify $ \sigma s(H)=s\sigma(H)$ by comparing the coefficients of each generator $E_I$. To check actions on $E_i$, for $i=1,\ldots,n-2$, one has 
\begin{align*}
    s\sigma(E_i)&=s(H-\sum_{\substack{k=2,k\ne i+1}}^{n-1} E_k)\\
    &=H-\sum_{\substack{k=2,k\ne i+1}}^{n-1} E_k-\sum_{\substack{I\subset\{1,\ldots,n-1\},\\1,i+1\notin I,\\|I|=\frac{n-1}{2},\ldots,n-4.}}E_I
    \phantom{=}+\sum_{\substack{I\subset\{1,\ldots,n-1\},\\1,i+1\in I,\\|I|=\frac{n-1}{2},\ldots,n-4.}}E_I.
\end{align*}
On the other hand, 
\begin{align*}
    \sigma s(E_i)&=\sigma(E_i+\sum_{\substack{I\subset\{1,\ldots,n-1\}, i\in I,\\|I|=\frac{n-1}{2},\ldots,n-4.}}E_{I})\\
    &=H-\sum_{\substack{I\subset\{2,\ldots,n-1\},i+1\notin I,\\|I|=1,\ldots,n-4}.}D_{I\cup\{n\}}+\sum_{\substack{I\subset\{2,\ldots,n-1\},\\|I|=\frac{n-1}2-1,\ldots,n-5.}}D_{I\cup\{1,n\}}\\
    &\phantom{=}+\sum_{\substack{I\subset\{2,\ldots,n-1\},i+1,\notin I\\|I|=2,\ldots,\frac{n-1}2-1.}}E_I.
\end{align*}
Similarly, one can check $\sigma s(E_i)=s\sigma(E_i)$ for $i\ne n-1$ by comparing the coefficients. Finally, one can verify
$$
s(\sigma(E_{n-1}))=s(E_{1})=\sigma(s(E_{n-1})).
$$
\end{rema}

\section{Computing cohomology}
\label{sect:comp}

In this section, we 
study the $G$-module 
$$
M=\Pic(\ocM_{0,n}), 
$$
and the quotient 
$
Q=M/N,
$
from \eqref{eqn:NQ}, for various $G\subset \fS_n$. 

\subsection*{Cohomological criteria}
We focus on two properties, which are necessary for linearizability
of a regular $G$-action on a smooth projective rational variety $X$, see, e.g.,  \cite[Proposition 2.5]{BogPro}:
\begin{itemize}
    \item[{\bf (H1) }] For all subgroups $G'\subset G$ one has 
    $$
    \rH^1(G', \Pic(X)) =\rH^1(G',\Pic(X)^*)=0.
    $$
    \item[{\bf (SP) }] The $G$-module $\Pic(X)$ is stably permutation. 
\end{itemize}
Since $\rH^1$ vanishes on permutation
modules, ${\bf (SP)}$ implies {\bf (H1)}, but the converse does not hold, in general. Computationally, it is easier to check {\bf (H1)}.

\begin{exam}
\label{exam:n=6}   
For $n=6$ and $G\subseteq \fS_6$, property {\bf (H1)} for the action on 
$M=\Pic(\ocM_{0,6})$ 
does not imply {\bf (SP)}, e.g., for the action of 
$$
G\simeq C_2\times C_4 :=\langle
(3, 4), (1, 2, 5, 6)
\rangle, 
$$
and 
$$
 G\simeq (C_2)^3 :=\langle   (1, 5)(2, 6), (3, 4),
    (1, 2)(5, 6)
\rangle,
$$
see the analysis in \cite[Section 6]{CTZ-segre}, as well as \cite[Section 4]{kun}. 
Furthermore, there are $G\subset \fS_6$
such that 
\begin{itemize}
\item 
$Q$ fails {\bf(H1)} but $M$ satisfies it, e.g., for $G=\langle (1,2)(3,4) (5,6)\rangle$, one has
$$
\rH^1(G,M)=0,\quad \rH^1(G,Q)=\bZ/2.
$$
Actually, $M$ is a permutation module while $Q$ is not. Indeed, under appropriate choices of basis, $M$ is of the form
$$
\bZ^4\oplus\bZ[C_2]^6, 
$$
and $Q$ is of the form
$$
\bZ\oplus\bZ[C_2]^2\oplus\bZ[e],
$$
where $G$ acts on $e$ via $-1$.
    \item Both $Q$ and $M$ fail {\bf (H1)}:
    all groups containing $G=C_2^2$
    from Proposition~\ref{prop:cohomo},
    in these cases we have 
    $$
    \rH^1(G,M)=\rH^1(G,Q)=\bZ/2.
    $$
\end{itemize}
\end{exam}


\subsection*{Statement of results}
\begin{prop}\label{prop:cohomo}
For $n_1,n_2,n_3 \in \bN$ with
$2(n_1+n_2+n_3)=n$
let 
$$
\iota_1= (1,2)\ldots (2n_1-1,2n_1)
(2(n_1+n_2)+1,2(n_1+n_2)+2)\ldots (n-1,n),
$$
$$
\!\!\iota_2= (2n_1+1,2n_1+2),\ldots,(2(n_1+n_2)-1,2(n_1+n_2))\ldots
(n-1,n),
$$
and put $G:=\langle \iota_1,\iota_2\rangle$. 
Then 
$$
\rH^1(G,M)=\bZ/2.
$$
\end{prop}
The first part of Theorem~\ref{thm:main} follows:
\begin{coro}
\label{coro:linear}
    For every even $n>5$ and every subgroup of $\fS_n$ containing $G$, the induced action on $\ocM_{0,n}$ is not stably linearizable. 
\end{coro}

For example, when $n_1=n_2=n_3=1$
$$
\iota_1=(12)(56), \quad \iota_2=(34)(56),
$$
and the corresponding action on $\ocM_{0,6}$, which is $\fS_6$-equivariantly birational to the Segre cubic, is not stably linearizable.

We apply the results above to rationality questions
over nonclosed fields, completing the proof of 
Theorem~\ref{thm:main}:
\begin{theo}
\label{thm:biquad}
Let $F$ be a field admitting a biquadratic extension. Then, for all even $n\ge 6$ there exist forms of $\ocM_{0,n}$ over $F$ that are not retract rational, and thus not stably rational, over $F$.    
\end{theo}

In particular, this yields nonrational forms over $F=\bC(t)$, a field with trivial Brauer group. 

\begin{proof}
    Indeed, let $G\simeq C_2^2$ be the group identified in Proposition~\ref{prop:cohomo}, with $\rH^1(G,\Pic(\ocM_{0,n})) = \bZ/2.$
    Let $\Gamma=\Gal(F'/F)$ be the Galois group of the biquadratic extension $F'/F$. 
    We construct a form $X$ of  $\ocM_{0,n}$ over $F$ such that $\Gamma$ acts on $\Pic(\overline{X})  = \Pic(\ocM_{0,n})$ via $G$.
    This gives an {\bf (H1)}-obstruction to retract rationality. 
\end{proof}

\subsection*{Proof of Proposition~\ref{prop:cohomo}}

Put
$$
\sigma:=\iota_1\iota_2=(1,2)\cdots(2(n_1+n_2)-1,2(n_1+n_2)),
$$
$$
\tau:=\iota_2=(2n_1+1,2n_1+2)\cdots(n-1,n), 
$$
so that $G=\langle\sigma,\tau\rangle$.
We will repeatedly use the inflation-restriction exact sequence 
\begin{equation}
\label{eqn:res}
0\to \rH^1(\langle\tau\rangle, A^\sigma) \to \rH^1(G,A)\to 
\rH^1(\langle\sigma\rangle, A)^\tau,
\end{equation}
with the usual notation for invariants under the actions of $\sigma,\tau$.

\

\noindent
{\em Step 1.} 
Observe that $M$ admits a decomposition, as a $G$-module, 
$$
M=L\oplus P,
$$
where $L$ consists of $\bZ$-linear combinations of $H$ and $E_I$, with $n-1\notin I$, and $P$ is generated, over $\bZ$, by $E_I$ with $n-1\in I$. We have
$$
\rH^1(G,M)=\rH^1(G,L)\oplus \rH^1(G,P).
$$

\

\noindent
{\em Step 2.}
The involution $\sigma$ is contained in $\fS_{n-1}$, permuting $(n-1)$ points and therefore linearizable.
Thus
$$
\rH^1(\langle \sigma\rangle,M) =\rH^1(\langle\sigma\rangle,L)=\rH^1(\langle\sigma\rangle,P)=0.
$$
Moreover, $P$ is a $G$-permutation module. Indeed, for $I$ with $n-1\in I$, $\sigma E_I=E_{\sigma (I)}\in P$, and $\tau E_I=E_{(\tau\cdot(n-1,n))(I)}\in P$.
It follows that 
$$
\rH^1(G,P) = 0,
$$
and
$$
\rH^1(G,M) =\rH^1(G,L) = \rH^1(\langle \tau\rangle, L^\sigma).
$$

\begin{rema}
Geometrically, cohomology is already contributed on the {\em toric model} $\oL_n$, 
obtained by blowing up $(n-2)$ general points on $\bP^{n-3}$. 
\end{rema}

\

\noindent
{\em Step 3.}
Let $N\subset L$ be the submodule of $\bZ$-linear combinations of $E_I$ with $|I|\geq 2$ and $n-1\notin I.$ We have a short exact sequence 
$$
0\to N\to L\to Q\to 0,
$$
of $G$-modules, with $Q$ generated by $H,E_1,\ldots, E_{n-2}$, modulo $N$, and the corresponding long exact sequence of $\langle \tau\rangle$-modules:
$$
0\to N^\sigma\to L^\sigma\to Q^\sigma\to \rH^1(\langle\sigma\rangle,N)\to\ldots
$$
Since 
$
\sigma(E_I)=E_{\sigma(I)},
$
the $\sigma$-action on $N$ yields naturally a permutation module, realized via permutation of indices of $E_I$. 
So 
$$
\rH^1(\langle\sigma\rangle,N)=0.
$$
The short exact sequence 
$$
0\to N^\sigma\to L^\sigma\to Q^\sigma\to 0
$$
gives rise to the long exact sequence
\begin{align}\label{eqn:longexactM'}
    \rH^1(\langle \tau\rangle,N^\sigma)\to \rH^1(\langle \tau\rangle,L^\sigma)\to \rH^1(\langle \tau\rangle,Q^\sigma)\to \rH^2(\langle \tau\rangle,N^\sigma).
\end{align}


\

\noindent
{\em Step 4.}
   The $\langle\tau\rangle$-module $N^\sigma$ has the form:
   $$
   N^\sigma = \bZ[\langle\tau\rangle] \oplus \cdots \oplus \bZ[\langle\tau\rangle].
   $$
   In particular,
   $$
   \rH^1(\langle \tau\rangle,N^\sigma)=\rH^2(\langle \tau\rangle,N^\sigma)=0.
   $$
 Indeed, a $\bZ$-basis of  $N^\sigma$ is given by 
    $$
    e_I:=\begin{cases}
        E_{I}+E_{\sigma(I)}&\text{if }\,\,\sigma(I)\ne I,\\
        E_{I} & \text{if }\,\,\sigma(I)=I,
    \end{cases}
    $$
    for
    $$
    I\subset\{1,2,\ldots, n-2\}, \quad 2\leq|I|\leq n-4.
    $$
   To show that $N^\sigma$ is a direct sum of copies of $\bZ[\langle\tau\rangle]$, it suffices to show that $\tau(e_I)=e_{I'}$, for some $I'\ne I$ and $e_I\ne e_{I'}$.  Observe that
   $$
   \sigma(I)^c=\sigma(I^c),\quad   I^c:=\{1,\ldots,n-2\}\setminus I.
   $$
   There are three cases:
   \begin{itemize}
       \item If $\sigma(I)=\tau(I)=I$, then
       $$
       \tau(e_I)=\tau(E_I)=D_{I\cup \{n-1\}}=E_{I^c}=e_{I^c}
       $$ 
       and thus $e_I\ne e_{I^c}$.
       \item If $\sigma(I)\ne I $ and $\tau(I)=I$,
       then
       \begin{align*}
\tau(e_I)&=\tau(E_I)+\tau(E_{\sigma(I)})=D_{I\cup \{n-1\}}+D_{\sigma(I)\cup \{n-1\}}\\
&=E_{I^c}+E_{\sigma(I)^c}=E_{I^c}+E_{\sigma(I^c)}=e_{I^c}.
\end{align*}
Since $I^c\ne I$ and $I^c\ne\sigma(I^c)$, we know  that $e_I\ne e_{I^c}$.
\item If $\tau(I)\ne I$, then $\sigma(I)\ne I$, and
\begin{align*}
\tau(e_I)&=E_{\tau(I)^c}+E_{(\tau\sigma(I))^c}=E_{\tau(I)^c}+E_{(\sigma\tau(I))^c}\\
&=E_{\tau(I)^c}+E_{\sigma(\tau(I)^c)}=e_{\tau(I)^c}.
\end{align*}
To be concrete, assume that $1\in I$ and $2\notin I$. Then $1\in\tau(I)^c$ and $1\notin\sigma(I)$, so that $\tau(I)^c\ne\sigma(I)$. Since $|I|\geq 2$, one can see that $\tau(I)^c\ne I$ and thus $e_{\tau(I)^c}\ne e_I$.
   \end{itemize}
In conclusion, $\tau(e_I)\ne e_{I}$, in all cases, and $N^\sigma$ is as claimed, and thus has vanishing first and second cohomology.
It follows that 
$$
\rH^1(\langle\tau\rangle,M^\sigma)=
\rH^1(\langle\tau\rangle,L^\sigma)=
\rH^1(\langle \tau\rangle,Q^\sigma).
$$

\noindent
{\em Step 5.}
To show that $\rH^1(\langle \tau\rangle,Q^\sigma)=\bZ/2$, 
let
$$
\Sigma_i:=\sum_{|I|=i} E_I,
$$ 
where the sum is over $I\subseteq\{1,2,\ldots,n-2\}$ with $|I|=i$. Put 
$
\Sigma:=\Sigma_1
$
and set
\begin{align*}
    &e_0 :=H-\Sigma, & 
    \\
    &e_i:=H-\Sigma+(E_{2i-1}+E_{2i}), &1\leq i \leq n_1+n_2,
    \\
    &w_j:=E_{2j-1}, &n_1+n_2+1\leq j \leq \frac{n-2}{2},
    \\
    &v_j:=H-\Sigma +E_{2j}, &n_1+n_2+1\leq j \leq \frac{n-2}{2}.
\end{align*}
    Then 
    $$
    \{e_i,w_j,v_j\}
    $$ 
    for $0\leq i\leq n_1+n_2$ and $n_1+n_2+1\leq j\leq\frac{n-2}{2}$ gives a $\bZ$-basis of $Q^\sigma$. Moreover, for $1\leq i\leq n_1+n_2$ and $n_1+n_2+1\leq j\leq\frac{n-2}{2}$, one has
    $$
    \tau(e_0)=-e_0,\quad \tau(e_i)=e_i,\quad \text{ and } \quad\tau(w_j)=v_j.
    $$
Indeed, $Q^\sigma$ is generated, over $\bZ$, by 
$$
H, (E_1+E_2), \ldots, (E_{2(n_1+n_2)-1}+ E_{2(n_1+n_2}),  E_{2(n_1+n_2)+1}, \ldots, E_{n-2}.
$$
We now show that $\{e_i,w_j,v_j\}$ gives another basis.
First, observe that
  \begin{align*}
      H-\Sigma=D_{34\ldots n}-(E_1+E_2)+\underbrace{\sum_{1,2\notin I, E_I\in N} E_I}_{\in N^\sigma}.
  \end{align*}
  Indeed, if $1,2\notin I$ and $E_I\in N$, $1,2\notin\sigma(I)$ and $E_{\sigma(I)}$ will also appear in the summand. Then  $\sigma(H-\Sigma)=H-\Sigma\pmod {N^\sigma}$ and 
  $$
    e_j,w_j,v_j\in Q^\sigma.
  $$
Moreover, $\{e_j,w_j,v_j\}$ generates $Q^\sigma$ since
  $$
  E_{2i-1}+E_{2i}=e_i-e_0,\quad E_{2j}=v_j-e_0
  $$
  and
  $$
  H=(\frac {4-n}2) e_0+\sum_{i=1}^{n_1+n_2}e_i+\sum_{j=n_1+n_2+1}^{\frac{n-2}2}\left(w_j+v_j\right).
  $$  
  To compute the $\tau$-action on this basis, one can first compute
    \begin{align*}
        H-\Sigma=&D_{34\ldots n}-(E_1+E_2)\pmod {N^\sigma}
        \\\stackrel{\tau}{\longmapsto}& D_{34\ldots n}-D_{1,n-1}-D_{2,n-1}\\
        =&D_{34\ldots n}-2H+2\Sigma-(E_1+E_2)\pmod{N^\sigma}\\
        =&H-\Sigma+(E_1+E_2)-2H+2\Sigma+(E_1+E_2)\pmod{N^\sigma}\\
        =&-H+\Sigma,
    \end{align*}
i.e., 
$$
\tau(e_0)=-e_0.
$$
Then we have
 \begin{align*}
        H-\Sigma+E_{2i-1}+& E_{2i}
        \stackrel{\tau}{\longmapsto} 
         -H+\Sigma+D_{2i-1,n-1}+D_{2i,n-1}\\
        =& -H+\Sigma+2H-2\Sigma+(E_{2i-1}+E_{2i})\pmod{N^\sigma}\\
        =& \,\,\, \,\,\,\,H-\Sigma+(E_{2i-1}+E_{2i}) \quad \quad \quad \quad \quad \,\,\pmod{N^\sigma}.
    \end{align*}
Note that the equalities hold for all $1\leq i\leq \frac n2$. In particular, 
$$
\tau(e_i)=e_i,\quad\text{for}\quad 1\leq i \leq n_1+n_2.
$$
Finally, 
\begin{align*}
        \tau(w_j)=D_{2_j,n-1}
        &=H-\Sigma+E_{2j}-\sum_{\substack{2j\notin I\\E_I\in N}}E_I\\
        &=H-\Sigma+E_{2j}\pmod{N^\sigma},
\end{align*}
 i.e.,
 $$
 \tau(w_j)=v_j,\quad\text{for}\quad n_1+n_2+1\leq j\leq \frac{n-2}{2}.
 $$
 In conclusion, 
 $$
 Q^\sigma=\bZ[e_0]+\sum_{i=1}^{n_1+n_2}\bZ[e_i]+\sum_{j=n_1+n_2+1}^{\frac{n-2}{2}}\bZ[w_j,v_j],
 $$
where $\tau$ acts trivially on $e_i$, permutes $w_j$ and $v_j$, and the unique $(-1)$-eigenvector $e_0$ contributes to 
$$
\rH^1(\langle\tau\rangle,Q^\sigma)=\bZ/2.
$$
This completes the proof of Proposition~\ref{prop:cohomo}.

\begin{rema} \label{rema:H1Qnonzero}
    Notice that when $n_1=n_2=0$, the argument above shows 
$$
\rH^1(C_2,Q)=\bZ/2,
$$
where the $C_2$ is generated by $(1,2)(3,4)\ldots(n-1,n)$. Computational experiments suggest that 
$$
\rH^1(H,M)=0,
$$
for all cyclic subgroups $H\subset \fS_n$.
\end{rema}


\subsection*{Small dimensional examples}

\ 

\noindent $\mathbf{n=6}$:
By Theorem 1 and the analysis in
Section 6 of \cite{CTZ-segre}, we know that 
the $G$-action on $\Pic(\ocM_{0,6})$ satisfies {\bf (SP)} iff 
the $G$-action is linearizable, thus, nonlinearizable actions are not stably linearizable, as they fail {\bf (SP)}.  

\begin{rema}
This indicates an error in the application 
in \cite[p.~295]{HT-torsor}: Proposition 21 there
asserts that the standard and non-standard
actions of $\mathfrak{A}_5$ are stably
birational, contradicting our cohomology
computation. The gap occurs in the sentence:
``However, for any finite group $G$ and 
automorphism $a : G \rightarrow G$, precomposing by $a$ yields an action on $G$-modules; this respects permutation and stably permutation modules.''
\end{rema}

\ 

\noindent $\mathbf{n=8}$:
There is a unique (conjugacy class of) $G'=C_2^2\subset \fS_8$ such that 
$$
\rH^1(G', \Pic(\ocM_{0,8}))=\bZ/2,
$$
and all 
$G\subseteq \fS_8$ failing {\bf (H1)} on $M$ contain $G'$. 
With {\tt magma}, we find: 

\begin{itemize}
\item There are 66 (conjugacy classes of) groups containing this $G'$.
\item Of the remaining 230 classes, 96 are contained in the (unique) $\fS_7\subset \fS_8$, the
corresponding actions are linearizable.
\item After that, there are 56 contained in the (unique) $\fS_6\times C_2$ -- these actions are birational to an action on a 5-dimensional torus; such actions have been analyzed, over nonclosed fields, in \cite{hoshi}.
\item We are left with 78 classes. Applying \cite[Algorithm F4]{hoshi} to these classes, we found at least $37$ classes of groups $G\subset \fS_8$ having vanishing cohomology but with $\Pic(\ocM_{0,8})$ 
failing the {\bf (SP)} condition. 
\item 
Among the 41 remaining classes, 13 leave invariant an odd cycle. These actions are stably linearizable by Proposition~\ref{prop:oddcycle}.

\item There are 28 remaining classes, including a minimal 
$$
C_2^2=\langle (1,2)(3,4)(5,6)(7,8), (1,3)(2,4)(5,7)(6,8)\rangle,
$$
which (up to conjugation) is 
contained in every remaining class. 
The action of this $C_2^2$ on $M$ 
yields a permutation module:
$$
\bZ[C_2^2]^{19}\oplus\bZ[C_2^2/C_2]^3\oplus\bZ[C_2^2/C_2']^3\oplus\bZ[C_2^2/C_2'']^3\oplus\bZ^5.
$$
However, on $Q$, this action fails {\bf (H1)}, and Theorem~\ref{thm:stable} is not applicable to any of these cases. 
\end{itemize}

\ 

\noindent $\mathbf{n=10}$:
We find more minimal groups contributing cohomology: 
$$
\rH^1(G,\Pic(\ocM_{0,10}))=\bZ/2
$$ 
when
\begin{itemize}
  \item $G=C_2^2=\langle (1,2)(3,4)(5,6)(7,8),(1,2)(9,10)\rangle$,
    \item $G=C_2^2=\langle (1,2)(3,4)(5,6),(5,6)(7,8)(9,10)\rangle$,
    \item $G=C_2\times C_4= \langle (3, 6)(8, 10),
    (1, 2)(5, 9),
    (1, 2)(3, 10, 6, 8)(4, 7)\rangle$,
    \item $G=\fD_4=\langle (3, 6)(8, 10),
    (1, 2)(5, 9)(8, 10),
    (1, 2)(3, 10, 6, 8)(4, 7)\rangle.$
\end{itemize}

\section{Three-dimensional case}
\label{sect:arithm}

Next, we give a criterion for rationality of the Segre cubic, exhibit forms failing stable rationality over arbitrary fields admitting a bi-quadratic extension, 
and establish stable rationality, provided $Q$ is stably permutation, for the action of the absolute Galois group.  

Recall that $X_6$ denotes the symmetrically linearized 
GIT quotient with equivalent presentations:
\begin{itemize}
\item{$(\bP^1)^6$ under the diagonal action of $\SL_2$; or}
\item{$\Gr(2,6)$ under the diagonal action of the torus
$\mathsf T \simeq \bG_m^5$.}
\end{itemize}
These have ten isolated nodes, the images of the $D_I, |I|=3$
under the blow down $\beta: \ocM_{0,6} \rightarrow X_6$.
These are classically embedded $X_6 \subset \bP^4$ as
cubic threefolds, known as Segre cubic threefolds \cite{CTZ-segre}. The remaining boundary divisors $D_I,|I|=2$
correspond to planes passing through four nodes.

\begin{theo}
\label{thm:segre-rat}
Let $X$ be a form of the Segre cubic threefold over a nonclosed field $F$ of characteristic zero, and $\tilde{X}$ its standard resolution of singularities, a form of $\ocM_{0,6}$. Then 
$X$ is rational over $F$ if and only if 
the Galois-module $\Pic(\ocM_{0,6})$ 
satisfies ${\bf (SP)}$. 
\end{theo}

\begin{proof}
This is closely related to the linearizability result \cite[Theorem 1]{CTZ-segre}. The 
group-theoretic analysis there shows that the only
cases where the Galois action on the Picard group is stably permutation
are:
\begin{itemize}
\item{when one of the ten nodes is Galois invariant;}
\item{the Galois action is contained in an 
$\fS_5$-action associated with permutations
of {\em five} of the marked points;}
\item{the Galois group acts via $C_2^2$, leaving 
three planes invariant, and the set of nodes splits into 
a union of five Galois orbits of length two. 
}
\end{itemize}
Note that the first two cases are easily shown
to be rational: Projecting from a node gives
a birational map to $\bP^3$, cf.~Example~\ref{exam:segp3}. 
And when the 
action factors through $\fS_5$, the moduli space
arises via the Kapranov construction, i.e., is
a blow-up of $\bP^3$. 

Recall that in the third case, the Galois action
factors through $\fS_2 \times \fS_4 \subset \fS_6$ corresponding
to a partition of the six points conjugate to 
$$\{1,2,3,4,5,6\} = \{3,4\} \cup \{1,2,5,6\}.$$
Our $C_2\times C_2$ action is conjugate to
$$\left<(34), (15)(26)\right> \subset \fS_6$$
This leaves the boundary divisors $D_{34},D_{15}$, and $D_{26}$
invariant. Identifying singular points
with the boundary divisors in $\ocM_{0,6}$,
the orbits are
\begin{align*}
&\{D_{123}=D_{456},D_{124}=D_{356}\},\quad 
\{D_{125}=D_{346}, D_{156}=D_{234} \},\\
&\{D_{126}=D_{345}, D_{256}=D_{134} \},\quad 
\{D_{135}=D_{246}, D_{145}=D_{236} \},\\
&\{D_{136}=D_{245}, D_{146}=D_{235}\}.
\end{align*}

We emphasize that the invariant divisor
classes reflect boundary divisors defined
over $F$. Indeed, our moduli space has 
$F$-rational smooth points so there is no obstruction to descending Galois-invariant
divisors.  

We claim this moduli space is birational over $F$ to a toric threefold, i.e.,
an equivariant compactification of a 
nonsplit torus over $F$.  

Consider the Losev-Manin moduli space associated to the partition above.
Specifically, points $3$ and $4$ are not
permitted to collide with other points
but points from $\{1,2,5,6\}$ may collide with one another. This is toric by Proposition~\ref{prop:twopointtoric}, i.e.,
the orbits of the homogeneous quartic forms
vanishing along $\{1,2,5,6\}$ modulo the
torus fixing $\{3,4\}$. This geometric description is compatible with the Galois action.

Rationality of three-dimensional toric 
varieties has been settled in 
\cite[Theorem~2]{kun}:
The variety is rational over $F$ iff
the Picard module is stably permutation
for the Galois action. 

Here is an alternative rationality construction: Pick one of the boundary divisors $D_I,|I|=2$ invariant under the
Galois action. With our choice of indexing
this could be $D_{34},D_{15}$, or $D_{26}$;
we take $D_{34}$.  
This corresponds to a plane $P\subset X$
containing four ordinary singularities, 
i.e., the images of $D_{34j}, j=1,2,5,6$.
We blow this plane up -- inducing a 
small resolution of the four singularities
-- and then blow down the proper transform
of the plane. This yields a complete
intersection of two quadrics 
$X_{2,2}\subset \bP^5$ 
with six singularities, the images of
the singularities of $X$ {\em not} 
contained in $P$.  Under the $C_2 \times C_2$
action, we have three orbits each with
two singular points. For each orbit, the line
joining the singularities is contained
in $X_{2,2}$.  Projecting from that line
gives 
$$X_{2,2} \stackrel{\sim}{\dashrightarrow}
\bP^3;$$
the birationality is classical 
cf.~\cite[Proposition~2.2]{CT-quad}.  
\end{proof}

\bibliographystyle{alpha}
\bibliography{coho}
\end{document}